\documentclass{llncs}

\usepackage{amssymb,amsmath,graphics,graphicx,color,amsfonts, pict2e}
\spnewtheorem{observation}{Observation}{\bf}{\it}
\spnewtheorem{claimm}{Claim}{\itshape}{\rmfamily}
\newcommand{\mim}[1]{M_{#1}}

\begin{document}
\pagestyle{plain}

\title{Maximal induced matchings in triangle-free graphs\thanks{This work is supported by the Research Council of Norway (project SCOPE, 197548/F20), and by the European Research Council under the European Union's Seventh Framework Programme (FP/2007-2013)/ERC Grant Agreement n.~267959.}}

\author{Manu Basavaraju \and Pinar Heggernes \and Pim van 't Hof \and Reza Saei \and Yngve Villanger}

\institute{
Department of Informatics, University of Bergen, Norway.
\texttt{\{manu.basavaraju,pinar.heggernes,pim.vanthof,reza.saeidinvar, yngve.villanger\}@ii.uib.no}\\
}

\maketitle

\begin{abstract}
An induced matching in a graph is a set of edges whose endpoints induce a $1$-regular subgraph. It is known that any $n$-vertex graph has at most $10^{n/5}\approx 1.5849^n$ maximal induced matchings, and this bound is best possible. We prove that any $n$-vertex triangle-free graph has at most $3^{n/3}\approx 1.4423^n$ maximal induced matchings, and this bound is attained by any disjoint union of copies of the complete bipartite graph $K_{3,3}$. Our result implies that all maximal induced matchings in an $n$-vertex triangle-free graph can be listed in time $O(1.4423^n)$, yielding the fastest known algorithm for finding a maximum induced matching in a triangle-free graph.
\end{abstract}

\section{Introduction}
A celebrated result due Moon and Moser~\cite{MM65} states that any graph on $n$ vertices has at most $3^{n/3}\approx 1.4423^n$ maximal independent sets. Moon and Moser also proved that this bound is best possible by characterizing the extremal graphs as follows: a graph on $n$ vertices has exactly $3^{n/3}$ maximal independent sets if and only if it is the disjoint union of $n/3$ triangles. Given the structure of these extremal graphs, it is natural to investigate how many maximal independent sets a triangle-free graph can have. Hujter and Tuza~\cite{HT93} showed that a triangle-free graph on $n$ vertices has at most $2^{n/2}\approx 1.4143^n$ maximal independent sets; this bound is attained by any 1-regular graph. Later, Byskov~\cite{Byskov} gave an algorithmic proof of the same result, along with more general results.

More recently, Gupta, Raman, and Saurabh~\cite{GRS} showed that for any fixed non-negative integer~$r$, there exists a constant $c<2$ such that any graph on $n$ vertices has at most $c^n$ maximal $r$-regular induced subgraphs. Their upper bound is tight when $r\in \{0,1,2\}$ and hence generalizes the aforementioned result by Moon and Moser. In particular, their result for $r=1$ shows that any $n$-vertex graph has at most $10^{n/5}\approx 1.5849^n$ maximal induced matchings, and this upper bound is attained by any disjoint union of complete graphs on five vertices. The structure of these extremal graphs again raises the question how much the upper bound can be improved for triangle-free graphs. We answer this question by proving the following result.

\begin{theorem}
\label{t-main}
Every triangle-free graph on $n$ vertices contains at most $3^{n/3}$ maximal induced matchings, and this bound is attained by any disjoint union of copies of~$K_{3,3}$.
\end{theorem}

We would like to mention some implications of the above theorem. There exist algorithms that list the maximal independent sets of any graph with polynomial delay~\cite{JohnsonP88,TsukiyamaIAS77}, which means that the time spent between the output of two successive maximal independent sets is polynomial in the size of the graph. Together with the aforementioned upper bounds on the number of maximal independent sets, this implies that the maximal independent sets of an $n$-vertex graph $G$ can be listed in time $O^*(3^{n/3})$, or in time $O^*(2^{n/2})$ in case $G$ is triangle-free.\footnote{We use the $O^*$-notation to suppress polynomial factors, i.e., we write $O^*(f(n))$ instead of $O(f(n)\cdot n^{O(1)})$ for any function $f$.}

Cameron~\cite{Cam89} observed that the maximal induced matchings of a graph $G$ are exactly the maximal independent sets in the square of the line graph of $G$. Consequently, the maximal induced matchings of any graph can be listed with polynomial delay. Combining this with the aforementioned upper bound by Gupta et al.~\cite{GRS} yields an algorithm for listing all maximal induced matchings of an $n$-vertex graph in time $O^*(10^{n/5})=O(1.5849^n)$. Gupta et al.~\cite{GRS} also obtained an algorithm for finding a maximum induced matching in an $n$-vertex graph in time $O(1.4786^n)$, which is the current fastest algorithm for solving this problem. Theorem~\ref{t-main} implies that we can do better on triangle-free graphs, as the following two results show. We point out that the problem of finding a maximum induced matching remains NP-hard on subcubic planar bipartite graphs~\cite{HOV13}, a small subclass of triangle-free graphs.

\begin{corollary}
For any triangle-free graph on $n$ vertices, all its maximal induced matchings can be listed in time $O^*(3^{n/3})=O(1.4423^n)$ with polynomial delay.
\end{corollary}

\begin{corollary}
For any triangle-free graph $G$ on $n$ vertices, a maximum induced matching in $G$ can be found in time $O^*(3^{n/3})=O(1.4423^n)$.
\end{corollary}

\section{Definitions and Notations}
All graphs we consider are finite, simple and undirected. We refer the reader to the monograph by Diestel~\cite{diestel} for graph terminology and notation not defined below. 

Let $G$ be a graph. For a vertex $v\in V(G)$, we write $N_G(v)$ and $N_G[v]$ to denote open and closed neighborhoods of $v$, respectively. Let $A \subseteq V(G)$. The closed neighborhood of $A$ is defined as $N_G[A]=\bigcup_{v\in A}N_G[v]$, and the open neighborhood of $A$ is $N_G(A) = N_G[A] \setminus A$. We write $G[A]$ to denote the subgraph of $G$ induced by $A$, and we write $G - A$ to denote the graph $G[V(G)\setminus A]$. If $A=\{v\}$, then we simply write $G-v$ instead of $G-\{v\}$. For any non-negative integer $r$, we say that $G$ is {\em $r$-regular} if the degree of every vertex in $G$ is~$r$. A $3$-regular graph is called {\em cubic}. A cycle $C$ with vertices $v_1,v_2\ldots,v_k$ and edges $v_1v_2,\ldots,v_{k-1}v_k,v_kv_1$ is denoted by $C=v_1v_2\cdots v_k$.

A {\it matching} in $G$ is a subset $M\subseteq E(G)$ such that no two edges in $M$ share an endpoint. For a matching $M$ in $G$ and a vertex $v\in V(G)$, we say that $M$ {\it covers} $v$ if $v$ is an endpoint of an edge in $M$. A matching $M$ is called {\it induced} if the subgraph induced by endpoints of the edges in $M$ is $1$-regular. An induced matching $M$ in $G$ is {\it maximal} if there exists no induced matching $M'$ in $G$ such that $M\subsetneq M'$. We write $\mim{G}$ to denote the set of all maximal induced matchings in $G$. Let $X$ and $Y$ be two disjoint subsets of $V(G)$. We define $\mim{G}(X,Y)$ to be the set of all maximal induced matchings of $G$ that cover no vertex of $X$ and every vertex of $Y$. Clearly, $\mim{G}=\mim{G}(\emptyset,\emptyset)$. When there is no ambiguity we omit subscripts from the notations. 

\section{Twins and Maximal Induced Matchings}

Let $G$ be a graph. Two vertices $u,v\in V(G)$ are {\it (false) twins} if $N_G(u)=N_G(v)$. In this paper, whenever we write twin, we mean false twin. For every vertex $u\in V(G)$, the {\it twin set} of $u$ is defined as $T_G(u)=\{v\in V(G) \mid N_G(u)=N_G(v) \}$, i.e., $T_G(u)$ consists of the vertex $u$ and all its twins. All the twin sets together form a partition of the vertex set of $G$, and we write $\tau(G)$ to denote the number of sets in this partition, i.e., $\tau(G)$ denotes the number of twin sets in $G$.

\begin{definition}
\label{d-operation1}
Let $G$ be a graph. For any two non-adjacent vertices $u,v\in V(G)$, we define $G_{u\rightarrow v}$ to be the graph obtained from $G$ by making $u$ into a twin of $v$ by deleting the edge $ux$ for every $x\in N_G(u)\setminus N_G(v)$ and adding the edge $uy$ for every $y\in N_G(v)\setminus N_G(u)$.
\end{definition}

The following lemma identifies certain pairs of vertices $u$ and $v$ for which the operation in Definition~\ref{d-operation1} does not decrease the number of maximal induced matchings in the graph. This lemma will play a crucial role in the proof of our main result. Note that this lemma holds for general graphs $G$, and not only for triangle-free graphs.

\begin{lemma}
\label{main lemma}
Let $G$ be a graph and let $u,v\in V(G)$. If no maximal induced matching in $G$ covers both $u$ and $v$, then $|M_{G_{u \rightarrow v}}|\geq |M_G|$ or $|M_{G_{v\rightarrow u}}|\geq |M_G|$.
\end{lemma}

\begin{proof}
Without loss of generality, we assume that the number of matchings in $M_G$ that cover $u$ is greater than or equal to the number of matchings in $M_G$ that cover $v$, i.e., $|M_G(\emptyset,\{u\})|\ge |M_G(\emptyset,\{v\})|$. Since every matching in $M_G$ that covers $u$ does not cover $v$ due to the assumption that $M_G(\emptyset,\{u,v\})=\emptyset$, it holds that $M_G(\emptyset,\{u\})=M_G(\{v\},\{u\})$. By symmetry, we also have that $M_G(\emptyset,\{v\})=M_G(\{u\},\{v\})$. This implies that $|M_G(\{v\},\{u\})|\geq |M_G(\{u\},\{v\})|$. We now use this fact to prove that $|M_{G_{v\rightarrow u}}|\geq |M_G|$.

For convenience, we write $G'=G_{v \rightarrow u}$.
The set $M_G$ of all maximal induced matchings in $G$ can be partitioned as follows:
$$\mim{G}=\mim{G}(\{v\},\{u\}) \uplus \mim{G}(\{u\},\{v\}) \uplus \mim{G}(\emptyset,\{u,v\}) \uplus \mim{G}(\{u,v\},\emptyset).$$
We can partition $\mim{G'}$ in the same way:
$$\mim{G'}=\mim{G'}(\{v\},\{u\}) \uplus \mim{G'}(\{u\},\{v\}) \uplus \mim{G'}(\emptyset,\{u,v\}) \uplus \mim{G'}(\{u,v\},\emptyset).$$

We claim that $\mim{G}(\{v\},\{u\}) = \mim{G'}(\{v\},\{u\})$. Let $M\in M_G(\{v\},\{u\})$. We claim that $M\in M_{G'}(\{v\},\{u\})$. It is easy to verify that $M$ is a induced matching in $G'$, as we only change edges incident with $v$ when transforming $G$ into $G'$, and $M$ does not cover $v$. For contradiction, suppose $M$ is not a maximal induced matching in $G'$. Then there is an edge $xy\in E(G')$ such that $M\cup \{xy\}$ is an induced matching in $G'$. Since $u$ and $v$ are twins in $G'$ and $M$ covers $u$, we find that $v\notin \{x,y\}$. This implies that $xy\in E(G)$, so $M\cup \{xy\}$ is a matching in~$G$ that does not cover $v$. In fact, $M\cup \{xy\}$ is an induced matching in $G$, since every edge in $E(G)\setminus E(G')$ is incident with $v$. This contradicts the maximality of $M$ in $G$. Hence we have that $\mim{G}(\{v\},\{u\}) \subseteq \mim{G'}(\{v\},\{u\})$. To show why $\mim{G'}(\{v\},\{u\}) \subseteq \mim{G}(\{v\},\{u\})$, let $M'\in M_{G'}(\{v\},\{u\})$. For similar reasons as before, $M'$ is an induced matching in $G$. To show that $M'$ is maximal in $G$, suppose for contradiction that there is an edge $xy\in E(G)$ such that $M'\cup \{xy\}$ is an induced matching in $G$. Then $v\notin \{x,y\}$, this time due to the assumption that no maximal induced matching in $G$ covers both $u$ and $v$. Now we can use similar arguments as before to conclude that $M'\cup \{x,y\}$ is an induced matching in $G'$, yielding the desired contradiction.

By assumption, we have $\mim{G}(\emptyset,\{u,v\})=\emptyset$. Since $u$ and $v$ are twins in $G'$ by construction, we also know that $\mim{G'}(\emptyset,\{u,v\})=\emptyset$ and $\mim{G'}(\{v\},\{u\}) = \mim{G'}(\{u\},\{v\})$. Recall that $|\mim{G}(\{v\},\{u\})| \ge |\mim{G}(\{u\},\{v\})|$. Hence we can infer that $|\mim{G'}(\{u\},\{v\})| \ge |\mim{G}(\{u\},\{v\})|$. Hence, in order to show that $|M_{G'}|\geq |M_G|$, it suffices to show that $|\mim{G'}(\{u,v\},\emptyset) |\ge |\mim{G}(\{u,v\},\emptyset)|$.

Let $M \in \mim{G}(\{u,v\},\emptyset)$. We claim that $M\in M_{G'}(\{u,v\},\emptyset)$. It is easy to see that $M$ is an induced matching in $G'$, as the only edges that are modified are incident with $v$ and $M$ does not cover $v$. Suppose, for contradiction, that $M$ is not a maximal induced matching in $G'$. Then there exists an edge $xy\in E(G')$ such that $M\cup \{xy\}$ is an induced matching in $G'$. If $v \notin \{x,y\}$, then $M\cup \{xy\}$ is also an induced matching in $G$, contradicting the maximality of $M$. Thus we have $v \in \{x,y\}$. Without loss of generality, suppose $x=v$. Let $M'=M \cup \{vy\}$. Now consider $M''=M \cup \{uy\}$.
Since $M'$ is induced matching and $u$ and $v$ are twins in $G'$, we infer that $M''$ is also an induced matching in $G'$. Note that the edge $uy$ is also present in $G$, so $M''$ is an induced matching in $G$. This contradicts the maximality of $M$, implying that $M \in \mim{G'}(\{u,v\},\emptyset)$ and consequently $\mim{G}(\{u,v\},\emptyset)\subseteq M_{G'}(\{u,v\},\emptyset)$. This completes the proof of Lemma~\ref{main lemma}.
\qed
\end{proof}

For our purposes, we need to extend Definition~\ref{d-operation1} as follows.

\begin{definition}
\label{d-operation2}
Let $G$ be a graph. For any two non-adjacent vertices $u,v\in V(G)$, the graph $G_{T_G(u)\rightarrow v}$ is the graph obtained from $G$ by making each vertex of $T_G(u)$ into a twin of $v$ as follows: for every $u'\in T_G(u)$, delete the edge $u'x$ for every $x\in N_G(u)\setminus N_G(v)$ and add the edge $u'y$ for every $y\in N_G(v)\setminus N_G(u)$.
\end{definition}

The following lemma is an immediate corollary of Lemma~\ref{main lemma}, since we can repeatedly apply the operation in Definition~\ref{d-operation1} on all the vertices in $T_G(u)$.

\begin{lemma}
\label{l-morematchings}
Let $G$ be a graph and let $u,v\in V(G)$. If no maximal induced matching in $G$ covers both $u$ and $v$, then $|M_{G_{T_G(u) \rightarrow v}}|\geq |M_G|$ or $|M_{G_{T_G(v)\rightarrow u}}|\geq |M_G|$.
\end{lemma}

We also need the following two lemmas in the proof of our main result.

\begin{lemma}
\label{l-trianglefree}
Let $G$ be a triangle-free graph. For any two non-adjacent vertices $u,v\in V(G)$, the graph $G_{T_G(u)\rightarrow v}$ is triangle-free.
\end{lemma}

\begin{proof}
Let $u,v\in V(G)$. For contradiction, suppose that $G_{T_G(u)\rightarrow v}$ contains a triangle $C$. Observe that every edge that was added to $G$ in order to create $G_{T_G(u)\rightarrow v}$ is incident with a vertex in $T_G(u)$ and a vertex in $N_G(v)\setminus N_G(u)$. Hence, $C$ contains an edge $u'x$ such that $u'\in T_G(u)$ and $x\in N_G(v)\setminus N_G(u)$. Let $y$ be the third vertex of $C$. Since $G$ is triangle-free, the set $N_G(v)$ forms an independent set in both $G$ and $G_{T_G(u)\rightarrow v}$. This implies in particular that $y$ is not adjacent to $v$ in $G_{T_G(u)\rightarrow v}$, and since we did not delete any edge incident with $v$ when creating $G_{T_G(u)\rightarrow v}$, it holds that $y$ is not adjacent to $v$ in $G$ either. Moreover, since both $u'$ and $y$ do not belong to $N_G(v)\setminus N_G(u)$, the edge $u'y$ is present in $G$. But then, by Definition~\ref{d-operation2}, the edge $u'y$ should have been deleted when $G$ was transformed into $G_{T_G(u)\rightarrow v}$. This yields the desired contradiction.
\qed
\end{proof}

\begin{lemma}\label{lem:twin}
Let $G$ be a triangle-free graph and let $u,v\in V(G)$ be two non-adjacent vertices. If $u$ and $v$ are not twins, then $\tau(G_{T_G(u)\rightarrow v})<\tau(G)$.
\end{lemma}

\begin{proof}
Suppose $u$ and $v$ are not twins. Then $T_G(u)$ and $T_G(v)$ are two different twin sets in $G$. By Definition~\ref{d-operation2}, the vertices of $T_G(u)\cup T_G(v)$ all belong to the same twin set in $G_{T_G(u)\rightarrow v}$, namely the twin set $T_{G_{T_G(u)\rightarrow v}}(u)=T_{G_{T_G(u)\rightarrow v}}(v)$. Let $x\in V(G)\setminus (T_G(u)\cup T_G(v))$. We prove that all the vertices in $T_G(x)$ belong to the same twin set in $G_{T_G(u)\rightarrow v}$, which implies that $\tau(G_{T_G(u)\rightarrow v})<\tau(G)$. 

Suppose there is a vertex $y\in T_G(x)$ such that $x$ and $y$ are not twins in $G_{T_G(u)\rightarrow v}$. Without loss of generality, suppose there is a vertex $z\in N_{G_{T_G(u)\rightarrow v}}(y)\setminus N_{G_{T_G(u)\rightarrow v}}(x)$. Since $x$ and $y$ are twins in $G$, we either have $xz,yz\in E(G)$ or $xz,yz\notin E(G)$. In the first case, the edge $xz$ is deleted from $G$ when $G_{T_G(u)\rightarrow v}$ is created, which implies that $x\in N_G(u)\setminus N_G(v)$ by Definition~\ref{d-operation2}. However, since $x$ and $y$ are twins in $G$, it holds that $y\in N_G(u)\setminus N_G(v)$ as well, implying that the edge $yz$ should not exist in $G_{T_G(u)\rightarrow v}$. This contradicts the definition of $z$. If $xz,yz\notin E(G)$, then we can use similar argument to conclude that $xz$ should be an edge in $G_{T_G(u)\rightarrow v}$, again yielding a contradiction.
\qed
\end{proof}

\section{Proof of Theorem~\ref{t-main}}

This section is devoted to proving Theorem~\ref{t-main}. We first prove that any triangle-free graph on $n$ vertices has at most $3^{n/3}$ maximal induced matchings. At the end of the section, we show why the bound in Theorem~\ref{t-main} is best possible.

A triangle-free graph on $n$ vertices that has more than $3^{n/3}$ maximal induced matchings is called a {\em counterexample}. For contradiction, let us assume that there exists a counterexample. Then there exists a counterexample $G$ such that for every counterexample $G'$, it holds that either $|V(G')|> |V(G)|$, or $|V(G')|=|V(G)|$ and $\tau(G')\geq \tau(G)$. Let $n=|V(G)|$. By definition of a counterexample, $|M_G|> 3^{n/3}$. We will prove a sequence of structural properties of $G$, and finally conclude that $G$ does not exist, yielding the desired contradiction.

\begin{lemma}
\label{l-basis}
$G$ is connected and has at least three vertices.
\end{lemma}

\begin{proof}
First assume for contradiction that $G$ is not connected. Let $G_1, G_2, \ldots, G_k$ denote the connected components of $G$. By the choice of $G$, none of the connected components of $G$ is a counterexample. Hence $|M_{G_i}|\leq 3^{|V(G_i)|/3}$ for each $i\in \{1,\ldots,k\}$. But then $|\mim{G}|=\prod_{i=1}^{k} |\mim{G_i}|\leq 3^{n/3}$, contradicting the assumption that $G$ is a counterexample.
\qed
\end{proof}

\begin{lemma}
\label{l-twinlemma}
Let $u,v\in V(G)$. If there is no maximal induced matching in $G$ that covers both $u$ and $v$, then $u$ and $v$ are twins.
\end{lemma}

\begin{proof}
Suppose there is no maximal induced matching in $G$ that covers both $u$ and $v$. In particular, this implies that $u$ and $v$ are not adjacent. Let $G'=G_{T_G(u)\rightarrow v}$ and $G''=G_{T_G(v)\rightarrow u}$. By Lemma~\ref{l-morematchings}, we have that $|M_{G'}|\geq |M_G|$ or $|M_{G''}|\geq |M_G|$. Without loss of generality, suppose $|M_{G'}|\geq |M_G|$. The graph $G'$ is triangle-free due to Lemma~\ref{l-trianglefree}. This, together with the fact that $|M_{G'}|\geq |M_G|>3^{n/3}$, implies that $G'$ is a counterexample. But by Lemma~\ref{lem:twin}, it holds that $\tau(G')<\tau(G)$, which contradicts the choice of $G$.
\qed
\end{proof}

\begin{lemma}\label{exc.inc.upperbound}
For every edge $uv\in E(G)$ and every set $X\subseteq V(G)\setminus \{u,v\}$, it holds that $|\mim{G}(X, \{u,v\})|\le 3^{(n-|X\cup N[\{u,v\}]|)/3}$.
\end{lemma}
\begin{proof}
Let $G'=G-(X\cup N_G[\{u,v\}])$. We first show that for every matching $M\in M_G(X,\{u,v\})$, it holds that $M\setminus \{uv\} \in M_{G'}$. Let $M\in M_G(X, \{u,v\})$. Since $uv\in E(G)$ and  $M$ covers both $u$ and $v$, the edge $uv$ belongs to $M$. Since $M$ does not cover any vertex in $X$, it is clear that the set $M'=M\setminus\{uv\}$ is an induced matching in $G'$. We show that $M'$ is maximal. For contradiction, suppose there exists an edge $xy\in E(G')$ such that $M'\cup \{xy\}$ is an induced matching in $G'$. Since neither $x$ nor $y$ belongs to the set $X\cup N_G[\{u,v\}]$, we have in particular that there is no edge between the sets $\{x,y\}$ and $\{u,v\}$. Hence, adding the edge $xy$ to $M$ yields an induced matching in $G$, contradicting the assumption that $M$ is a maximal induced matching in~$G$. 

We now know that for every matching $M\in M_G(X,\{u,v\})$, it holds that $M\setminus \{uv\} \in M_{G'}$. Note that, for any two matchings $M_1,M_2\in M_G(X, \{u,v\})$ with $M_1\neq M_2$, the sets $M_1\setminus \{uv\}$ and $M_2\setminus \{uv\}$ are not equal, as both $M_1$ and $M_2$ contain the edge $uv$. Hence we have that $|M_G(X,\{u,v\})| \leq |M_{G'}|$. Since $G'$ has less vertices than $G$ and is thus not a counterexample, we have that $|\mim{G'}|\le 3^{|V(G')|/3}=3^{(n-|X\cup N[\{u,v\}]|)/3}$. We conclude that $|\mim{G}(X, \{u,v\})|\le |M_{G'}|\leq  3^{(n-|X\cup N[\{u,v\}]|){3}}$.
\qed
\end{proof}

\begin{lemma}\label{degree1}
$G$ has no vertex of degree less than~$2$.
\end{lemma}
\begin{proof}
By Lemma~\ref{l-basis}, the graph $G$ is connected and $n\geq 3$. Hence, $G$ has no vertices of degree~$0$. Assume for contradiction that $G$ contains a vertex $v$ with $d(v)=1$. Let $u$ be the unique neighbor of $v$. If $G$ is a star, then $M_{G}=E(G)$, implying that $|M_G|=n-1\leq 3^{n/3}$. Since this contradicts the fact that $G$ is a counterexample, we infer that $G$ is not a star. Since $G$ is connected and triangle-free, $u$ has a neighbor $w$ with $d(w)\geq 2$. Note that there is no maximal induced matching in $G$ that covers both $v$ and $w$. Then $u$ and $w$ must be twins due to Lemma~\ref{l-twinlemma}. This is a contradiction, as $d(v)<d(w)$ implies that $v$ and $w$ cannot be twins.
\qed \end{proof}

\begin{lemma}\label{2degree2in5cycle}
$G$ has no $5$-cycle containing two non-adjacent vertices of degree~$2$.
\end{lemma}
\begin{proof}
For contradiction, suppose there is a 5-cycle containing two non-adjacent vertices $u$ and $v$ such that $d(u)=d(v)=2$. Clearly, the vertices $u$ and $v$ are not twins, and there is no maximal induced matching in $G$ that covers both $u$ and $v$. This contradicts Lemma~\ref{l-twinlemma}.
\qed
\end{proof}

\begin{lemma}\label{1degree2inc4}
$G$ has no $4$-cycle containing exactly one vertex of degree~$2$.
\end{lemma}
\begin{proof}
Assume, for contradiction, that there exists a $4$-cycle $C=uvwx$ such that $d(u)=2$ and the other vertices of $C$ have degree more than~$2$. Then $u$ and $w$ are not twins, and there is no maximal induced matching in $G$ that covers both $u$ and $w$. This contradicts Lemma~\ref{l-twinlemma}. 
\qed
\end{proof}

\begin{lemma}\label{2degree2adjacent}
$G$ has no two adjacent vertices of degree $2$.
\end{lemma}

\begin{proof}
For contradiction, suppose there are two vertices $u$ and $v$ such that $d(u)=d(v)=2$ and $uv\in E(G)$. Let $a$ and $b$ denote the other neighbors of $u$ and $v$, respectively. Since $G$ is triangle-free, we have that $a\neq b$. We first show that $ab\notin E(G)$. For contradiction, assume that $ab\in E(G)$ and both $a$ and $b$ have degree~$2$. Then $G$ is isomorphic to $C_4$, implying that $|\mim{G}|=4\le 3^{4/3}$. This contradicts the fact that $G$ is a counterexample. Hence $a$ or $b$ has degree more than~$2$. Assume without loss of generality that $d(a)\ge 3$. Then $a$ and $v$ are not twins, and there is no matching in $M_G$ covering both $a$ and $v$. This contradiction to Lemma~\ref{l-twinlemma} implies that $ab\notin E$.

We now partition $\mim{G}$ into three sets $M(\emptyset, \{a\})$, $M(\{a\},\{b\})$, and $M(\{a,b\},\emptyset)$,  and find an upper bound on the size of each of these sets.

We first consider $M(\emptyset, \{a\})$. It is clear that $|M(\emptyset,\{a\})|=\sum_{p\in N(a)}|M(\emptyset,\{a,p\})|$. Let $p=u$. Since $|N[\{a,u\}]|=d(a)+2$, from Lemma~\ref{exc.inc.upperbound} we have $|M(\emptyset, \{a,u\})|\le 3^{(n-(d(a)+2))/{3}}$. Now consider the case that $p\neq u$. In this case, $|N[\{a,p\}]|=d(a)+d(p)$ and $d(p)\ge 2$ due to Lemma~\ref{degree1}, and thus Lemma~\ref{exc.inc.upperbound} implies $|M(\emptyset, \{a,p\})|\le 3^{(n-(d(a)+2))/{3}}$. Consequently, we obtain 
$$|M(\emptyset, \{a\})|=\sum_{p\in N(a)}|M(\emptyset,\{a,p\})|\le d(a) \cdot 3^{\frac{n-(d(a)+2)}{3}}.$$

We now find an upper bound on $|M(\{a\},\{b\})|$. Since no matching in $M(\{a\},\{b\})$ covers $u$, it holds that $M(\{a\},\{b\}) = M(\{a,u\},\{b\})$. Observe that $|M(\{a,u\},\{b\})|=\sum_{q\in N(b)}|M(\{a,u\},\{b,q\})|$. If $q=v$, then $|M(\{a,u\}, \{b,v\})|\le 3^{(n-(d(b)+3))/{3}}$ due to Lemma~\ref{exc.inc.upperbound} and the fact that $d(v)=2$ and $a\notin N[\{b,v\}]$. Let now $q\neq v$. First suppose $q$ is adjacent to $a$. Then $qauvb$ is a $5$-cycle, and hence Lemma~\ref{2degree2in5cycle} implies that $d(q)\geq 3$. Consequently, $|N[\{b,q\}]|=d(b)+d(q)\geq d(b)+3$, and since $u\notin N[\{b,q\}]$, we find that $|M(\{a,u\},\{b,q\})|\leq 3^{(n-(d(b)+4))/3}$ due to Lemma~\ref{exc.inc.upperbound}. Now suppose that $q$ is not adjacent to $a$. Then $N[\{b,q\}]$ contains neither $a$ nor $u$. Hence, Lemma~\ref{exc.inc.upperbound} and the fact that $|N[\{b,q\}]\geq d(b)+2$ imply that $|M(\{a,u\}, \{b,q\})|\le 3^{(n-(d(b)+4))/{3}}$. We conclude that
$$|M(\{a\}, \{b\})|\le  3^{\frac{n-(d(b)+3)}{3}} + (d(b)-1) \cdot 3^{\frac{n-(d(b)+4)}{3}} \, .$$

Finally, we consider $M(\{a,b\},\emptyset)$. Every matching in $M(\{a,b\},\emptyset)$ is maximal and covers neither $a$ nor $b$, so it must contain edge $uv$. Hence, $M(\{a,b\},\emptyset)=M(\{a,b\}, \{u,v\})$. Since $|N[\{u,v\}]|=4$, Lemma~\ref{exc.inc.upperbound} gives
$$|M(\{a,b\},\emptyset)|=|M(\{a,b\},\{u,v\})|\le 3^{\frac{n-4}{3}}.$$

Combining the obtained upper bounds, we find that
\begin{equation*}
\label{eq:case2}
|\mim{G}|\le f(d(a),d(b))\cdot 3^{\frac{n}{3}} \, ,
\end{equation*}
where the function $f$ is defined as follows:
\begin{equation*}
f(d(a),d(b)) = d(a) \cdot 3^{-\frac{d(a)+2}{3}}+ 3^{-\frac{d(b)+3}{3}}+  (d(b) -1) \cdot 3^{-\frac{d(b)+4}{3}} + 3^{-\frac{4}{3}} \, .
\end{equation*}
Recall that both $a$ and $b$ have degree at least~$2$ due to Lemma~\ref{degree1}. We observe that $f(2,2)<0.965$, yielding an upper bound of $0.965 \cdot 3^{n/3}$ on $|M_G|$ in case $d(a)=d(b)=2$. Now consider the case where $d(a)=2$ and $d(b)\geq 3$. Then the function $f$ is decreasing with respect to $d(b)$. Since $f(2,3)<0.959$, we find that $|M_G|< 0.959 \cdot 3^{n/3}$ in this case. By using similar arguments, we find that $|M_G|< 0.984 \cdot 3^{n/3}$ when $d(b)=2$ and $d(a)\geq 3$. Finally, when both $d(a)\geq 3$ and $d(b)\geq 3$, then the function $f$ is decreasing with respect to both variables $d(a)$ and $d(b)$ and is maximum when $d(a)=d(b)=3$. Since $f(3,3)<0.978$, we find that $|M_G|<0.978 \cdot 3^{n/3}$ whenever $d(a)\geq 3$ and $d(b)\geq 3$. Summarizing, we obtain a contradiction to the assumption that $|\mim{G}| > 3^{n/3}$ in each case, which completes the proof of this case.
\qed
\end{proof}

\begin{lemma}\label{degree2}
Let $u\in V(G)$. If $u$ has degree~$2$, then both its neighbors have degree~$3$.
\end{lemma}
 
\begin{proof}
Suppose $u$ has degree~$2$, and let $N(u)=\{v,w\}$. Then $d(v)\ge 3$ and $d(w)\ge 3$ due to Lemmas~\ref{degree1} and~\ref{2degree2adjacent}. For contradiction, suppose one of the neighbors of $u$, say $w$, has degree more than~$3$. Note that $M_G=M(\{u\},\emptyset) \uplus M(\emptyset,\{u,v\})\uplus M(\emptyset,\{u,w\})$. Since $M(\{u\},\emptyset)=\mim{G-u}$ and $G-u$ is not a counterexample due to the choice of $G$, we have $|M(\{u\},\emptyset)|\le 3^{(n-1)/{3}}$. Since $|N[\{u,v\}]|=d(v)+2$, we can use Lemma~\ref{exc.inc.upperbound} to deduce that $|M(\emptyset,\{u,v\})|\leq 3^{(n-(d(v)+2))/{3}}$ and $|M(\emptyset,\{u,w\})|\le 3^{(n-(d(w)+2))/{3}}$. Hence we find that 
\[ |M_G|\leq 3^{\frac{n-1}{3}} + 3^{\frac{n-(d(v)+2)}{3}} + 3^{\frac{n-(d(w)+2)}{3}} \, .\]
Since $d(v)\ge 3$ and $d(w)\ge 4$, we get 
$$|\mim{G}|\le 3^{\frac{n-1}{3}}+3^{\frac{n-5}{3}}+ 3^{\frac{n-6}{3}}=\big{(}3^{-\frac{1}{3}}+3^{-\frac{5}{3}}+3^{-\frac{6}{3}} \big{)} \cdot 3^{\frac{n}{3}} < 3^{\frac{n}{3}} \, ,$$
yielding the desired contradiction.
\qed
\end{proof}

\begin{lemma}\label{maxdegree4}
$G$ has no vertex of degree more than~$4$.
\end{lemma}
\begin{proof}
Suppose there is a vertex $u\in V(G)$ such that $d(u)\ge 5$. Due to Lemmas~\ref{degree1} and~\ref{degree2}, every neighbor of $u$ has degree at least $3$. Clearly, $\mim{G}=M(\emptyset,\{u\}) \uplus M(\{u\},\emptyset)$. To find an upper bound on $|M(\emptyset,\{u\})|$, observe that $|M(\emptyset,\{u\})|=\sum_{p\in N(u)}|M(\emptyset, \{u,p\})|$.
For every $p\in N(u)$, it holds that $|N[\{u,p\}]|=d(u)+d(p)$ and $d(p)\ge 3$. Hence, using Lemma~\ref{exc.inc.upperbound}, we find that 
$$
 |M(\emptyset,\{u\})|\leq d(u)\cdot 3^{\frac{n-(d(u)+3)}{3}}  \, .
$$
Since $M(\{u\},\emptyset)=\mim{G-u}$ and $G-u$ is not a counterexample, we have that 
$$
 |M(\{u\},\emptyset)|\le 3^{\frac{n-1}{3}} \, .
$$

Combining the two upper bounds yields 
$$|\mim{G}|\le d(u) \cdot 3^{\frac{n-(d(u)+3)}{3}}+3^{\frac{n-1}{3}} = \big{(} d(u)\cdot  3^{-\frac{d(u)+3}{3}}+3^{-\frac{1}{3}}\big{)} \cdot 3^{\frac{n}{3}} \, .$$
Since $d(u)\geq 5$ by assumption and $d(u)\cdot 3^{-(d(u)+3)/3}+3^{-1/3}<1$ whenever $d(u)\geq 5$, we conclude that $|M_G|<3^{n/3}$. This contradicts the fact that $G$ is a counterexample.
\qed
\end{proof}

\begin{lemma}\label{nodegree2inc4}
$G$ has no $4$-cycle containing a vertex of degree~$2$.
\end{lemma}

\begin{proof}
Assume for contradiction that $G$ has a $4$-cycle $C$ containing a vertex $u$ of degree~$2$. Due to Lemmas~\ref{1degree2inc4} and~Lemma~\ref{2degree2adjacent}, there is exactly one other vertex $v$ in $C$ that has degree~$2$, and $u$ and $v$ are not adjacent. Let $w$ and $y$ be the other two vertices of $C$. Since $d(u)=d(v)=2$, Lemma~\ref{degree2} implies that $d(w)=d(y)=3$. Let $x$ and $z$ be the neighbors of $w$ and $y$, respectively, that do not belong to $C$. We claim that $x\neq z$. For contradiction, suppose $x=z$. Then there is no maximal induced matching in $G$ that covers both $u$ and $x$. Hence, due to Lemma~\ref{l-twinlemma}, vertices $u$ and $x$ are twins. In particular $d(x)=2$, which implies that $V(G)=\{u,v,w,y,x\}$ and $|M_G|=|E(G)|=6<3^{5/3}$, contradicting the fact that $G$ is a counterexample.

Observe that $d(z)\in \{2,3,4\}$ due to Lemmas~\ref{degree1} and~\ref{maxdegree4}, and $d(x)\geq 2$ due to Lemma~\ref{degree1}. In order to find an upper bound on the number of maximal induced matchings in $G$, we partition $M_G$ as follows:
\begin{equation} 
\label{eq:3partition}
M_G= M(\emptyset,\{w\}) \uplus M(\{w\},\{z\})\uplus M(\{w,z\},\emptyset) \, .
\end{equation}

Since $N(w)=\{u,v,x\}$, it holds that $M(\emptyset,\{w\})=M(\emptyset,\{w,u\})\uplus M(\emptyset,\{w,v\}) \uplus M(\emptyset,\{w,x\})$. Recall that $d(x)\geq 2$. For every $p\in \{u,v,x\}$, it holds that $|N[\{w,p\}]|\geq 5$ and consequently $|M(\emptyset,\{w,p\})|\le 3^{(n-5)/{3}}$ due to Lemma~\ref{exc.inc.upperbound}. Therefore,
\[|M(\emptyset,\{w\})|\le 3\cdot3^{\frac{n-5}{3}} \, .\]

We now consider $M(\{w,z\},\emptyset)$.
Note that every maximal induced matching of $G$ that covers neither $w$ nor $z$ must contain either $uy$ or $vy$. Hence $|M(\{w,z\},\emptyset)|=|M(\{w,z\},\{u,y\})|+|M(\{w,z\},\{v,y\})|$.
Since $|N[\{u,y\}]|=|N[\{v,y\}]|=5$, we can use Lemma~\ref{exc.inc.upperbound} to find that
$$|M(\{w,z\},\emptyset)|\le 2\cdot3^{\frac{n-5}{3}}.$$

It remains to find an upper bound on $|M(\{w\},\{z\})|$. It is clear that $|M(\{w\},\{z\})|=\sum_{q\in N(z)} |M(\{w\},\{q,z\})|$. We first consider $M(\{w\},\{y,z\})$. Since $|N[\{y,z\}]|=d(y)+d(z)=3+d(z)$ and $w\notin N[\{y,z\}]$, we have $|M(\{w\},\{z,y\})|\leq 3^{(n-(d(z)+4))/3}$ due to Lemma~\ref{exc.inc.upperbound}. Now let $q\in N(z)\setminus \{y\}$. Observe that every matching in $M(\{w\},\{q,z\})$ contains edge $qz$ and does not cover $w$ by definition, and hence covers neither $u$ nor $v$. This means that $M(\{w\},\{q,z\})=M(\{u,v,w\},\{q,z\})$. Since $d(q)\geq 2$ due to Lemma~\ref{degree1}, it holds that $|N[\{q,z\}]|\geq d(z)+2$. If $q\neq x$, then $\{u,v,w\}\cap N[\{q,z\}]=\emptyset$ and hence $|\{u,v,w\}\cup N[\{q,z\}]|\geq d(z)+5$. Suppose $q=x$. Then $wvyzx$ is a $5$-cycle, and $d(q)\geq 3$ as a result of Lemma~\ref{2degree2in5cycle}. Moreover, although now $w\in N[\{q,z\}]$, neither $u$ nor $v$ belongs to $N[\{q,z\}]$. Hence $|\{u,v,w\}\cup N[\{q,z\}]|\geq d(z)+5$ also in this case. Using Lemma~\ref{exc.inc.upperbound}, we conclude that $|M(\{w\},\{q,z\})|=|M(\{u,v,w\},\{q,z\})|\leq 3^{n-(d(z)+5)/3}$. Since this holds for every $q\in N(z)\setminus \{y\}$, we find that
$$ |M(\{w\},\{z\})|\leq 3^{\frac{n-(d(z)+4)}{3}} + (d(z)-1) \cdot 3^{\frac{n-(d(z)+5)}{3}} \, .$$

The obtained upper bounds on $|M(\emptyset,\{w\})|$, $|M(\{w\},\{z\})|$, and $|M(\{w,z\},\emptyset)|$, together with~\eqref{eq:3partition}, yield the following inequality:
$$|\mim{G}|\le \big{(} 5\cdot3^{-\frac{5}{3}} + 3^{-\frac{d(z)+4}{3}} + (d(z)-1)\cdot 3^{-\frac{d(z)+5}{3}} \big{)} \cdot 3^{\frac{n}{3}}\, .$$
Recall that $d(z)\in \{2,3,4\}$. Since it can readily be verified that for each value of $d(z)\in \{2,3,4\}$, the above inequality simplifies to $|M_G|<3^{n/3}$, we obtain the desired contradiction.
 \qed
\end{proof}

\begin{lemma}\label{nodegree2}
 $G$ has no vertex of degree $2$.
\end{lemma}

\begin{proof}
Due to Lemma~\ref{degree2}, in order to prove Lemma~\ref{nodegree2}, it suffices to prove that there is no vertex of degree~$3$ in $G$ that is adjacent to a vertex of degree~$2$. For contradiction, suppose $G$ has a vertex $u$ such that $d(u)=3$ and $u$ is adjacent to at least one vertex of degree~$2$. Let $N(u)=\{v,w,x\}$. We distinguish two cases, depending on the number of vertices of degree~$2$ in the neighborhood of $u$.

\medskip
\noindent
{\em Case 1. $u$ has at least two neighbors of degree~$2$.}

\smallskip
\noindent
Without loss of generality, assume that $d(v)=d(x)=2$. Let $N(x)=\{u,t\}$. Observe that $d(t)=3$ due to Lemma~\ref{degree2}, and $d(w)\in \{2,3,4\}$ due to Lemmas~\ref{degree1} and~\ref{maxdegree4}. It is easy to see that we can partition $M_G$ as follows:
$$M_G = M(\emptyset,\{v\}) \uplus M(\{v\},\{w\}) \uplus M(\{v,w\},\{t\}) \uplus M(\{v,w,t\},\emptyset) \, .$$

First consider $M(\emptyset,\{v\})$. Since every matching in this set contains exactly one of the two edges incident with $v$, and both neighbors of $v$ have degree exactly~$3$ due to Lemma~\ref{degree2}, we can use Lemma~\ref{exc.inc.upperbound} to find that 
\[ |M(\emptyset,\{v\})|\leq 2 \cdot 3^{\frac{n-5}{3}} \, .\]

Now consider $M(\{v\},\{w\})$. It is clear that $|M(\{v\},\{w\})|=|M(\{v\},\{u,w\})|+\sum_{q\in N(w)\setminus \{u\}} |M(\{v\},\{q,w\})|$. Lemma~\ref{exc.inc.upperbound}, together with the fact that $|N[\{u,w\}]|=d(u)+d(w)=3+d(w)$, implies that $|M(\{v\},\{u,w\})|\leq 3^{(n-(d(w)+3))/3}$. Let $q\in N(w)\setminus \{u\}$. Since $v$ has degree~$2$ and is therefore not contained in a $4$-cycle due to Lemma~\ref{nodegree2inc4}, vertex $q$ is not adjacent to $v$. Hence $|\{v\}\cup N[\{q,w\}]|=1+d(q)+d(w)\geq 3+d(w)$, where we use the fact that $d(q)\geq 2$ due to Lemma~\ref{degree1}. This implies that $|M(\{v\},\{q,w\})|\leq 3^{(n-(d(w)+3))/3}$. Since this holds for any $q\in N(w)\setminus \{u\}$, we find that
\[ |M(\{v\},\{w\})|\leq d(w)\cdot 3^{\frac{n-(d(w)+3)}{3}} \, .\]

\begin{sloppy}
To find an upper bound on $|M(\{v,w\},\{t\})|$, we first observe that $M(\{v,w\},\{t\})=M(\{u,v,w\},\{t\})$, as any maximal induced matching that covers neither $v$ nor $w$ but covers $t$, cannot cover $u$. Note that $|M(\{u,v,w\},\{t\})|=|M(\{u,v,w\},\{t,x\})| + \sum_{q\in N(t)\setminus \{x\}} |M(\{u,v,w\},\{t,q\})|$. Recall that $d(x)=2$ and $d(t)=3$. Since $G$ is triangle-free, $x$ is adjacent to neither $v$ nor $w$. The same holds for $t$ due to Lemma~\ref{nodegree2inc4} and the fact that $x$ has degree~$2$. Hence $|\{u,v,w\}\cup N[\{t,x\}]|=7$, so $|M(\{u,v,w\},\{t,x\})|\leq 3^{(n-7)/3}$ due to Lemma~\ref{exc.inc.upperbound}. 
Let $q\in N(t)\setminus \{x\}$. Then $q\notin \{u,v,w\}$ due to the triangle-freeness of $G$ and Lemma~\ref{nodegree2inc4}. Moreover, neither $u$ nor $v$ is adjacent to $q$ as a result of Lemmas~\ref{nodegree2inc4} and~\ref{2degree2in5cycle}, respectively. Recall that $d(q)\geq 2$ due to Lemma~\ref{degree1}. Moreover, if $w$ is adjacent to $q$, then $q$ has degree at least~$3$ by Lemma~\ref{2degree2in5cycle}. Hence $|\{u,v,w\}\cup N[\{t,q\}]|\geq 8$, so Lemma~\ref{exc.inc.upperbound} implies that $|M(\{u,v,w\},\{t,q\})|\leq 3^{(n-8)/3}$. Since $|N(t)\setminus \{x\}|=2$, we conclude that
\[ |M(\{v,w\},\{t\})|\leq 3^{\frac{n-7}{3}} + 2\cdot 3^{\frac{n-8}{3}} \, . \]
\end{sloppy}

Finally, we consider $M(\{v,w,t\},\emptyset)$. Since every matching in this set contains edge $ux$, we have that $M(\{v,w,t\},\emptyset)=M(\{v,w,t\},\{u,x
\})$. Using Lemma~\ref{exc.inc.upperbound} and the fact that $|N[\{u,x\}]|=5$, we deduce that
\[ |M(\{v,w,t\},\emptyset)| \leq 3^{\frac{n-5}{3}} \, .\]

Putting all this together, we obtain the following inequality:
\[ |M_G| \leq 3 \cdot 3^{\frac{n-5}{3}} + d(w)\cdot 3^{\frac{n-(d(w)+3)}{3}} + 3^{\frac{n-7}{3}} + 2\cdot 3^{\frac{n-8}{3}} \, .\]
It is easy to verify that the right-hand side of this inequality is less than~$3^{n/3}$ for every fixed value of $d(w)\in \{2,3,4\}$. This contradicts the assumption that $G$ is a counterexample and completes the proof of Case 1.

\medskip
\noindent
{\em Case 2. $u$ has exactly one neighbor of degree~$2$.}

\smallskip
\noindent
Without loss of generality, assume that $d(x)=2$. Then $d(v)\geq 3$ and $d(w)\geq 3$ due to Lemma~\ref{degree1}.
Let $N(x)=\{u,t\}$. We partition $M_G$ as follows: 
$$M_G=M(\emptyset,\{t\})\uplus M(\{t\},\{w\})\uplus M(\{t,w\},\{v\})\uplus M(\{t,w,v\},\emptyset) \, .$$

We first consider $M(\emptyset,\{t\})$. Due to Lemma~\ref{degree2}, vertex $t$ has degree~3. If $t$ has at least~$2$ neighbors of degree~$2$, then we can apply Case $1$ to vertex $t$ to obtain a contradiction. Suppose $t$ has at most one neighbor of degree~$2$. Since $x$ has degree~$2$, both vertices in $N(t)\setminus \{x\}$ have degree at least $3$. Hence $|N[\{t,x\}]=5$ and $|N[\{t,q\}]|\geq 6$ for every $q\in N(t)\setminus \{x\}$, and we can apply Lemma~\ref{exc.inc.upperbound} to find that 
$$|M(\emptyset, \{t\})|\le 3^{\frac{n-5}{3}}+2\cdot3^{\frac{n-6}{3}}\, .$$

To find an upper bound on $|M(\{t\},\{w\})|$, we first observe that $M(\{t\},\{w\})=M(\{t,x\},\{w\})$ due to the fact that no matching in $M(\{t\},\{w\})$ covers $x$. It is easy to see that $|M(\{t,x\},\{w\})|=|M(\{t,x\},\{w,u\})|+\sum_{q\in N(w)\setminus \{u\}}|M(\{t,x\},\{w,q\})|$. Since $x$ has degree $2$, it does not belong to any $4$-cycle due to Lemma~\ref{nodegree2inc4}. This implies that $t\notin N[\{u,w\}]$, and hence $|\{t,x\}\cup N[\{w,u\}]|\geq d(w)+d(u)+1=d(w)+4$. Applying Lemma~\ref{exc.inc.upperbound} yields $|M(\{t,x\},\{u,w\})|\le 3^{(n-(d(w)+4))/{3}}$. Let $q\in N(w)\setminus \{u\}$. By Lemma~\ref{degree1}, vertex $q$ has degree at least~$2$. Note that $q\neq t$ and $x\notin N(q)$ due to Lemma~\ref{nodegree2inc4}, and $d(q)\geq 3$ if $t\in N(q)$ due to Lemma~\ref{2degree2in5cycle}. This implies that $|\{t,x\}\cup N[\{w,q\}]|\geq d(w)+4$. By Lemma~\ref{exc.inc.upperbound}, we have that $|M(\{t,x\},\{w,q\})|\leq 3^{(n-(d(w)+4))/3}$. Hence we conclude that
$$|M(\{t\},\{w\})| = |M(\{t,x\},\{w\})|\le d(w) \cdot 3^{\frac{n-(d(w)+4)}{3}}\, .$$

Now consider $M(\{t,w\},\{v\})$. No matching in this set covers $x$ and therefore we have $M(\{t,w\},\{v\})=M(\{t,w,x\},\{v\})$. Clearly it holds that
$|M(\{t,w,x\},\{v\})|=|M(\{t,w,x\},\{v,u\})|+\sum_{q\in N(v)\setminus \{u\}}|M(\{t,w,x\},\{v,q\})|$. Observe that $t\notin N[\{v,u\}]$ due to Lemma~\ref{nodegree2inc4} and the fact that $G$ is triangle-free. Hence $\{t,w,x\}\cup |N[\{v,u\}]|=d(v)+d(u)+1= d(v)+4$, so we can apply Lemma~\ref{exc.inc.upperbound} to find that $|M(\{t,w,x\},\{v,u\})|\leq 3^{(n-(d(v)+4))/3}$. Let $q\in N(v)\setminus \{u\}$. Due to the triangle-freeness of $G$ and Lemma~\ref{nodegree2inc4}, vertex $x$ does not belong to $N[\{v,q\}]$, and neither $t$ nor $w$ belongs to $N[v]$. We claim that $|\{t,w,x\} \cup N[\{v,q\}]|\geq d(v)+5$. This is immediately clear if neither $t$ nor $w$ belongs to $N[q]$, as $d(q)\geq 2$. Suppose both $t$ and $w$ belong to $N[q]$. If $d(q)\leq 3$, then there is no maximal induced matching in $G$ that covers both $u$ and $q$. Since $u$ and $q$ are not twins, this contradicts Lemma~\ref{l-twinlemma}. Hence $d(q)\geq 4$, implying that $|\{t,w,x\} \cup N[\{v,q\}]|\geq d(v)+5$. If exactly one of the vertices $t$ and $w$ belongs to $N[q]$, then $d(q)\geq 3$ as a result of Lemma~\ref{2degree2in5cycle} and Lemma~\ref{nodegree2inc4}, respectively. Hence we have that $|\{t,w,x\} \cup N[\{v,q\}]|\geq d(v)+5$ also in this case. We can now invoke Lemma~\ref{exc.inc.upperbound} to find that $\sum_{q\in N(v)\setminus \{u\}}|M(\{t,w,x\},\{v,q\})|\leq (d(v)-1)\cdot 3^{(n-(d(v)+5))/3}$, and we can thus conclude that
$$|M(\{t,w\},\{v\})|= |M(\{t,w,x\},\{v\})|\le 3^{\frac{n-(d(v)+4)}{3}}+(d(v)-1)\cdot 3^{\frac{n-(d(v)+5)}{3}}\, .$$

Finally, we consider $M(\{t,w,v\},\emptyset)$. Since any maximal induced matching in this set must contain edge $ux$, it holds that $M(\{t,w,v\},\emptyset)=M(\{t,w,v\},\{u,x\})$. The fact that $|N[\{u,x\}]|=5$ together with Lemma~\ref{exc.inc.upperbound} readily implies that 
$$|M(\{t,w,v\},\emptyset)|=|M(\{t,w,v\},\{u,x\})|\le 3^{(n-5)/3}\, .$$

Combining the obtained upper bounds yields the following inequality:
$$|M_G|\le 2\cdot 3^{\frac{n-5}{3}}+2\cdot3^{\frac{n-6}{3}}+d(w) \cdot 3^{\frac{n-(d(w)+4)}{3}}+3^{\frac{n-(d(v)+4)}{3}}+(d(v)-1)\cdot 3^{\frac{n-(d(v)+5)}{3}} \, .$$
Recall that $d(v)\geq 3$ and $d(w)\geq 3$. We also have that $d(v)\leq 4$ and $d(w)\leq 4$ as a result of Lemma~\ref{maxdegree4}. It is therefore easy to check that the right-hand side of this inequality is less than~$3^{n/3}$. This contradicts the fact that $G$ is a counterexample and completes the proof of the lemma.
\qed
\end{proof}

\begin{lemma}\label{nodegree4}
$G$ is cubic.
\end{lemma}

\begin{proof}
Due to Lemmas~\ref{degree1},~\ref{maxdegree4}, and~\ref{nodegree2}, every vertex in $G$ has degree~$3$ or~$4$. Hence, in order to prove Lemma~\ref{nodegree4}, it suffices to prove that $G$ has no vertex of degree~$4$. For contradiction, suppose there exists a vertex $u$ such that $d(u)=4$. Let $v$ be a neighbor of $u$. To find an upper bound on $|M_G|$, we partition $M_G$ into two sets $M(\emptyset,\{v\})$ and $M(\{v\}, \emptyset)$ and find an upper bound on the sizes of these sets.

We first consider $M(\emptyset,\{v\})$. Observe that $|M(\emptyset,\{v\})|=\sum_{q\in N(v)}|M(\emptyset,\{v,q\})|$. If $q=u$, then $|N[\{v,q\}]|=d(v)+4$ and hence $|M(\emptyset,\{u,v\})|\le 3^{(n-(d(v)+4))/3}$ by Lemma~\ref{exc.inc.upperbound}. For any vertex $q\in N(v)\setminus\{u\}$, the fact that $|N[\{q,v\}]|\geq d(v)+3$ together with Lemma~\ref{exc.inc.upperbound} implies that $|M(\emptyset,\{q,v\})|\le 3^{(n-(d(v)+3))/3}$. Hence we find that
$$|M(\emptyset,\{v\})|\le 3^{\frac{n-(d(v)+4)}{3}}+(d(v)-1)\cdot 3^{\frac{n-(d(v)+3)}{3}}\, .$$

Since $M(\{v\},\emptyset)=\mim{G-v}$ and $G-v$ is not a counterexample, we have that 
$$|M(\{v\},\emptyset)|\le 3^{\frac{n-1}{3}}\, .$$

Hence we conclude that 
$$|\mim{G}|\le 3^{\frac{n-(d(v)+4)}{3}}+(d(v)-1)\cdot 3^{\frac{n-(d(v)+3)}{3}}+3^{\frac{n-1}{3}}\, .$$
For any fixed value of $d(v)\in \{3,4\}$, it can easily be verified that $|\mim{G}|\le 3^{n/3}$, yielding the desired contradiction.
\qed
\end{proof}

\begin{lemma}
\label{l-c5twin}
Let $u,v\in V(G)$. If $u$ and $v$ are contained in a $5$-cycle $C$, then $u$ and $v$ have no common neighbor in $V(G)\setminus V(C)$.
\end{lemma}

\begin{proof}
Suppose there is a $5$-cycle $C$ containing both $u$ and $v$. Then $u$ and $v$ are not twins. If $u$ and $v$ are adjacent, then they have no common neighbor due to the fact that $G$ is triangle-free. Suppose $u$ and $v$ are non-adjacent, and, for contradiction, assume there is a vertex $x\in V(G)\setminus V(C)$ such that $x$ is adjacent to both $u$ and $v$. Since $G$ is cubic due to Lemma~\ref{nodegree4}, there is no maximal induced matching in $G$ that covers both $u$ and $v$. Hence $u$ and $v$ must be twins due to Lemma~\ref{l-twinlemma}, yielding the desired contradiction.
\qed
\end{proof}

\begin{lemma}\label{onec4}
 $G$ contains at least one $4$-cycle.
\end{lemma}

\begin{proof}
Assume for contradiction that $G$ contains no $4$-cycle. Let $u$ be an arbitrary vertex in $G$. Recall that $G$ is cubic due to Lemma~\ref{nodegree4}. Let $N(u)=\{v,w,x\}$ and $N(x)=\{u,s,t\}$.
We consider the following partition of $M_G$:
\[  M(\emptyset,\{v\}) \uplus M(\{v\},\{w\}) \uplus M(\{v,w\},\{s\}) \uplus M(\{v,w,s\},\{t\}) \uplus M(\{v,w,s,t\},\emptyset) \, . \]

Since $d(v)=3$ and the closed neighborhood of any edge incident with $v$ consists of six vertices, Lemma~\ref{exc.inc.upperbound} implies that
$$|M(\emptyset,\{v\})|\le 3\cdot3^{\frac{n-6}{3}}\, .$$

Let us consider $M(\{v\},\{w\})$. Note that $|M(\{v\},\{w\})|=\sum_{q\in N(w)}|M(\{v\},\{w,q\})|$. If $q=u$, then $|M(\{v\},\{w,q\})|\leq 3^{(n-6)/3}$ due to Lemma~\ref{exc.inc.upperbound} and that fact that $|N[\{w,q\}]|=6$. Now suppose $q\in N(w)\setminus \{u\}$. The assumption that there is no $4$-cycle in $G$ implies that $v\notin N[\{w,q\}]$, and consequenly $|M(\{v\},\{w,q\})|\le 3^{(n-7)/3}$ due to Lemma~\ref{exc.inc.upperbound}. We therefore have that  
 $$|M(\{v\},\{w\})|\le 3^{\frac{n-6}{3}}+2\cdot3^{\frac{n-7}{3}}\, .$$

We now find an upper bound on $|M(\{v,w\},\{s\})|$. Let $M\in M(\{v,w\},\{s\})$. Observe that $M$ covers neither $v$ nor $w$, but covers $s$. Since $M$ is an induced matching, it cannot cover $u$. This implies that $M(\{v,w\},\{s\})=M(\{u,v,w\},\{s\})$. Clearly, $|M(\{u,v,w\},\{s\})|=\sum_{q\in N(s)}|M(\{u,v,w\},\{s,q\})|$.
 
We claim that $|\{u,v,w\}\cup N[\{s,q\}]|\geq 8$ for any $q\in N(s)$. Suppose $q=x$. Since $G$ is triangle-free and has no $4$-cycles by assumption, neither $v$ nor $w$ belongs to $N[\{s,q\}]$, so $|\{u,v,w\}\cup N[\{s,q\}]|= 8$ in this case. Now suppose $q\in N(s)\setminus \{x\}$. Since $G$ has no triangles and no $4$-cycles, none of the vertices in $\{u,v,w\}$ belongs to $N[s]$. For the same reason, $u\notin N[q]$ and $q$ is not adjacent to both $v$ and $w$. Hence $|\{u,v,w\}\cup N[\{s,q\}]|\geq 8$ also in this case. Lemma~\ref{exc.inc.upperbound} now implies that
$$|M(\{v,w\},\{s\})|=|M(\{u,v,w\},\{s\})|\le 3\cdot3^{\frac{n-8}{3}}\, .$$

  Let us now consider $M(\{v,w,s\},\{t\})$. Clearly, it holds that $|M(\{v,w,s\},\{t\})|=\sum_{q\in N(t)}|M(\{v,w,s\},\{t,q\})|$. 
  Similar to the previous paragraph, we can use the assumption that $G$ contains neither triangles nor $4$-cycles to deduce that 
  $$|M(\{v,w,s\},\{t\})|= |M(\{u,v,w,s\},\{t\})| \le 3^{\frac{n-8}{3}}+2\cdot3^{\frac{n-9}{3}} \, .$$

It remains to consider $M(\{v,w,s,t\},\emptyset)$. Since any maximal induced matching in this set contains edge $ux$, we have that $M(\{v,w,s,t\},\emptyset)=M(\{v,w,s,t\},\{u,x\})$. Since $|N[\{u,x\}]|=6$, Lemma~\ref{exc.inc.upperbound} implies that 
  $$|M(\{v,w,p,q\},\emptyset)|=|M(\{v,w,p,q\},\{u,x\})|\le 3^{\frac{n-6}{3}}.$$

We conclude that
  $$
  |\mim{G}|\le 5\cdot3^{\frac{n-6}{3}}+2\cdot3^{\frac{n-7}{3}}+4\cdot3^{\frac{n-8}{3}}+2\cdot3^{\frac{n-9}{3}} < 3^{\frac{n}{3}} \, ,
  $$
yielding the desired contradiction.
 \qed
\end{proof}

\begin{lemma}\label{lastlemma}
$G$ is isomorphic to $K_{3,3}$.
\end{lemma}

\begin{proof}
Let $uv$ be an edge of $G$ such that no edge in $E(G)\setminus \{uv\}$ is contained in more $4$-cycles than $uv$ is. Since $u$ and $v$ are adjacent and $G$ is triangle-free, $u$ and $v$ have no common neighbor. Recall that $G$ is cubic due to Lemma~\ref{nodegree4}. Let $N(u)=\{a,d\}$ and $N(v)=\{b,c\}$. It is easy to see that edge $uv$ is contained in at most four $4$-cycles. 

If $uv$ is contained in exactly four $4$-cycles, then $G$ is isomorphic to $K_{3,3}$ and the lemma holds. Suppose $uv$ is contained in at most three $4$-cycles. Due to Lemma~\ref{onec4} and the choice of $uv$, edge $uv$ belongs to at least one $4$-cycle. Hence, there is at least one edge between  sets $\{a,d\}$ and $\{b,c\}$. Note that $ad\notin E(G)$ and $bc\notin E(G)$, as $G$ is triangle-free. We distinguish four cases, depending on the adjacencies between vertices in $\{a,d\}$ and $\{b,c\}$.
 
 \begin{itemize}
    \item[] {\em Case 1: $ab \in E(G)$ and $ac,db,dc\notin E(G)$.}
   \item[] {\em Case 2: $ab,ac\in E(G)$ and $db,dc\notin E(G)$.}
 \item[] {\em Case 3: $ab,cd\in E(G)$ and $ac,db\notin E(G)$.}
  \item[] {\em Case 4: $ab,ac,bd\in E(G)$ and $dc\notin E(G)$.}
 \end{itemize}
Note that $uv$ belongs to exactly one $4$-cycle in Case 1, to exactly two $4$-cycles in Cases 2 and 3, and to exactly three $4$-cycles in Case 4.

 We partiton $M_G$ into five sets as follows: 
 \begin{equation*}
 \label{eq:final}
 M_G=M(\emptyset,\{a\}) \uplus M(\{a\},\{b\}) \uplus M(\{a,b\},\{c\}) \uplus M(\{a,b,c\},\{d\}) \uplus M(\{a,b,c,d\},\emptyset) \,.
 \end{equation*}
In Claims 1--5 below, we prove upper bounds on the sizes of the five sets on the right-hand side of the above inequality. We then combine these five upper bounds in order to obtain an upper bound on $|M_G|$.

\medskip
\noindent
{\em Claim 1. $|M(\emptyset,\{a\})|\le 3\cdot 3^{(n-6)/3}$.}

\smallskip
\noindent
 Since $G$ is cubic due to Lemma~\ref{nodegree4}, the closed neighborhood of any of the three edges incident with $a$ consists of six vertices. Hence Lemma~\ref{exc.inc.upperbound} ensures that $|M(\emptyset,\{a,q\})|\leq 3^{(n-6)/3}$ for every $q\in N(a)$, implying the upper bound given in Claim~1.

\medskip
\noindent
{\em Claim 2. $|M(\{a\},\{b\})|\le 2\cdot 3^{(n-6)/3}$.}

\smallskip
\noindent
Note that in all four cases, $ab$ belongs to $E(G)$. By definition, there is no matching in $M(\{a\},\{b\})$ that contains $ab$. For any of the other two edges incident with $b$, its closed neighborhood has size~$6$. Hence the correctness of the claimed upper bound again follows from Lemma~\ref{exc.inc.upperbound}.

\medskip
\noindent
{\em Claim 3. $|M(\{a,b\},\{c\})|\le 3^{(n-6)/3}+3^{(n-7)/3}$.}

\smallskip
\noindent
 First we consider Case~1. In this case, the closed neighborhood of $cv$ contains vertex $b$ and it does not contain $a$. 
 Therefore, $|\{a,b\}\cup N[\{c,v\}]|=7$ and consequently, by Lemma~\ref{exc.inc.upperbound} we have that $|M(\{a,b\},\{c,v\})|\le 3^{(n-7)/3}$. 
 Let $cq$ be one of the other edges incident with $c$. Recall that $uv$ belongs to exactly one $4$-cycle in Case 1. Hence, vertex $q$ is not adjacent to $b$, as otherwise $bv$ belongs to two $4$-cycles, contradicting the choice of $uv$. We claim that $q$ is not adjacent to $a$. For contradiction, suppose $q$ is adjacent to $a$. Then $qabvc$ is a $5$-cycle containing $a$ and $v$, so $a$ and $v$ cannot have a common neighbor in $V(G)\setminus \{q,a,b,v,c\}$ due to Lemma~\ref{l-c5twin}. The fact that both $a$ and $v$ are adjacent to $u$ gives the desired contradiction. Hence, for any $q\in N(c)\setminus \{v\}$, we have that $|\{a,b\}\cup N[\{c,q\}]|=8$, and thus Lemma~\ref{exc.inc.upperbound} implies that $|M(\{a,b\},\{c,q\})|\le 3^{(n-8)/3}$. We obtain that $|M(\{a,b\},\{c\})|\le 3^{(n-7)/3}+2\cdot 3^{(n-8)/3}$, which is strictly smaller than $3^{(n-6)/3}+3^{(n-7)/3}$.
 
Let us now consider Case $2$. Observe that no matching in $M(\{a,b\},\{c\})$ contains edge $ac$. Hence $|M(\{a,b\},\{c\})|=|M(\{a,b\},\{c,v\})|+|M(\{a,b\},\{c,q\})|$, where $q$ is the neighbor of $c$ other than $v$ and $a$. Both vertices $a$ and $b$ belong to $N[\{c,v\}]$ and hence $|\{a,b\}\cup N[\{c,v\}]|=6$. Therefore, Lemma~\ref{exc.inc.upperbound} guarantees that $|M(\{a,b\},\{c,v\})|\le 3^{(n-6)/3}$. Note that $a\notin N[\{c,q\}]$. We claim that $b\notin N[\{c,q\}]$. For contradiction, suppose $b\in N[\{c,q\}]$. Then $b$ is adjacent to $q$, and hence $bv$ belongs to three $4$-cycles, namely $bvua$, $bvcq$, and $bvca$. Since $uv$ belongs to only two $4$-cycles in Case 2, this contradicts the choice of $uv$. Hence, we have that $|\{a,b\}\cup N[\{c,q\}]|=7$ and consequently $|M(\{a,b\},\{c,q\})|\le 3^{(n-7)/3}$ by Lemma~\ref{exc.inc.upperbound}. We conclude that $|M(\{a,b\},\{c\})|\le 3^{(n-6)/3}+3^{(n-7)/3}$.
 
For Case $3$, let $q$ be the neighbor of $c$ other than $d$ and $v$. Since $b\in N[\{c,v\}]$ and $a\notin N[\{c,v\}]$ in Case 3, we have that $|\{a,b\}\cup N[\{c,v\}]|=7$ and therefore $|M(\{a,b\},\{c,v\})|\le 3^{(n-7)/3}$ due to Lemma~\ref{exc.inc.upperbound}. Moreover, since $N[\{c,d\}]$ contains neither $a$ nor $b$, Lemma~\ref{exc.inc.upperbound} implies that $|M(\{a,b\},\{c,d\})|\le 3^{(n-8)/3}$. We now consider edge $cq$. Since no matching in $M(\{a,b\},\{c,q\})$ covers $u$, it holds that $M(\{a,b\},\{c,q\})=M(\{a,b,u\},\{c,q\})$. Recall that $N(u)=\{v,a,d\}$, so $q\notin N[u]$. For contradiction, suppose $a\in N[q]$. Then $qauvc$ is a $5$-cycle containing two vertices, namely $u$ and $c$, that have a common neighbor, namely $d$, in the set $V(G)\setminus \{q,a,u,v,c\}$. This contradicts Lemma~\ref{l-c5twin}, so we conclude that $a\notin N[q]$. Consequently, we have that $|\{a,b,u\}\cup N[\{c,p\}]|=8$, so Lemma~\ref{exc.inc.upperbound} implies that $|M(\{a,b\},\{c,p\})|\le 3^{(n-8)/3}$. We conclude that $|M(\{a,b\},\{c\})|\le 3^{(n-7)/3}+2\cdot 3^{(n-8)/3}<3^{(n-6)/3}+3^{(n-7)/3}$.

Finally, we consider Case $4$. Let $N(c=\{a,v,q\}$. Since no matching in $M(\{a,b\},\{c\})$ covers $a$, we have that $|M(\{a,b\},\{c\})|=|M(\{a,b\},\{c,v\})|+|M(\{a,b\},\{c,q\})|$. The fact that $|N[\{c,v\}]|=6$ together with Lemma~\ref{exc.inc.upperbound} readily implies that $|M(\{a,b\},\{c,v\})|\le 3^{(n-6)/3}$. Since $N(a)=\{u,b,c\}$ and $N(b)=\{v,a,c\}$ in this case, neither $a$ nor $b$ belongs to $N[\{c,q\}]$. Hence, $|\{a,b\}\cup N[\{c,q\}]|=8$ and using Lemma~\ref{exc.inc.upperbound}, we deduce that $|M(\{a,b\},\{c,q\})|\le 3^{(n-8)/3}$. We can therefore conclude that $|M(\{a,b\},\{c\})|\le 3^{(n-6)/3}+3^{(n-8)/3}< 3^{(n-6)/3}+3^{(n-7)/3}$.
 
\medskip
\noindent
{\em Claim 4. $|M(\{a,b,c\},\{d\})|\le 3^{(n-7)/3}+3^{(n-8)/3}$.}

\smallskip
\noindent
 First we consider Cases $1$ and $2$. In both of these cases, the closed neighborhood of $du$ contains $a$, but neither $b$ nor $c$ belong to $N[\{d,u\}]$. Therefore, $|\{a,b,c\}\cup N[\{d,u\}]|=8$ and consequently, by Lemma~\ref{exc.inc.upperbound}, we have that $|M(\{a,b,c\},\{d,u\})|\le 3^{(n-8)/3}$. 
Let $q\in N(d)\setminus \{u\}$. Since no matching in $M(\{a,b,c\},\{d,q\})$ covers $v$, we have that $M(\{a,b,c\},\{d,q\})=M(\{a,b,c,v\},\{d,q\})$. Edge $au$ belongs to all the $4$-cycles to which $uv$ belongs. By the choice of $uv$, edge $au$ does not belong to any other $4$-cycle. In particular, the vertices $\{a,q,d,u\}$ do not induce a $C_4$, which implies that $a\notin N[q]$. We claim that $b$ is also not adjacent to $q$. For contradiction, suppose $b\in N[q]$. Then vertices $b$ and $u$ are contained in a $5$-cycle, namely $bpduv$, so the fact that they are adjacent to $a$ contradicts Lemma~\ref{l-c5twin}. Finally, we observe that $v\notin N[q]$, since $N(v)=\{u,b,c\}$. From this, we deduce that $|\{a,b,c,v\}\cup N[\{d,q\}]|=9$, and hence $|M(\{a,b,c\},\{d,q\})|\le 3^{(n-9)/3}$ due to Lemma~\ref{exc.inc.upperbound}. We conclude that $|M(\{a,b,c\},\{d\})|\le 3^{(n-8)/3}+2\cdot 3^{(n-9)/3}$, which is less than $3^{(n-7)/3}+3^{(n-8)/3}$.
 
 Now we consider Case $3$. Since no matching in $M(\{a,b,c\},\{d\})$ contains edge $dc$, we have that $|M(\{a,b,c\},\{d\})|=|M(\{a,b,c\},\{d,u\})|+|M(\{a,b,c\},\{d,q\})|$, where $q$ is the neighbor of $d$ other than $c$ and $u$. In Case 3, both vertices $a$ and $c$ are in $N[\{d,u\}]$ and $b\notin N[\{d,u\}]$. Hence $|\{a,b,c\}\cup N[\{d,u\}]|=7$, and therefore Lemma~\ref{exc.inc.upperbound} implies that 
 $|M(\{a,b,c\},\{d,u\})|\le 3^{(n-7)/3}$. From the observation that no matching in $M(\{a,b,c\},\{d\})$ covers $v$, it follows that $M(\{a,b,c\},\{d,q\})=M(\{a,b,c,v\},\{d,q\})$. We claim that neither $v$ nor $b$ belong to $N[\{d,q\}]$. The fact that neither $v$ nor $b$ belongs to $N[d]$ follows from the triangle-freeness of $G$ and the fact that we are in Case 3. Moreover, since $N(v)=\{u,b,c\}$, we have that $v\notin N[q]$. For contradiction, suppose $b\in N[q]$. Then the vertices $\{q,b,v,u,d\}$ induce a $5$-cycle. Vertices $b$ and $u$ lie on this $5$-cycle and have a common neighbor, namely $a$, in $V(G)\setminus \{q,b,v,u,d\}$. This contradiction to Lemma~\ref{l-c5twin} implies that $\{v,b\}\cap N[\{d,q\}]=\emptyset$. Hence $|\{a,b,c,v\}\cup N[\{d,q\}]|\geq 8$, and we can use Lemma~\ref{exc.inc.upperbound} to find that $M(\{a,b,c\},\{d,q\})|\le 3^{(n-8)/3}$. Consequently, we have that $|M(\{a,b,c\},\{d\})|\le 3^{(n-7)/3}+3^{(n-8)/3}$.
 
It remains to consider Case $4$. Let $q$ be the neighbor of $d$ other than $b$ and $u$. Since edge $db$ is not contained in any of the matchings in $M(\{a,b,c\},\{d\})$, we have that $|M(\{a,b,c\},\{d\}|=|M(\{a,b,c\},\{d,u\})|+|M(\{a,b,c\},\{d,q\})|$. Since $c\notin N[\{d,u\}]$, we have that $|\{a,b,c\}\cup N[\{d,u\}]|=7$ and hence $|M(\{a,b,c\},\{d,u\})|\le 3^{(n-7)/3}$ due to Lemma~\ref{exc.inc.upperbound}. Observe that vertex $v$ is not covered by any matching in $M(\{a,b,c\},\{d,q\})$, which implies that $M(\{a,b,c\},\{d,q\})=M(\{a,b,c,v\},\{d,q\})$. Since $N(a)=N(v)=\{u,b,c\}$, we have that neither $a$ nor $v$ belongs to $N[\{d,q\}]$. We now show that $c$ does not belong to $N[\{d,q\}]$ either. For contradiction, suppose otherwise. Since $c$ is not adjacent to $d$ in Case 4, vertex $c$ must be adjacent to $q$. Hence the vertices $\{c,q,d,u,v\}$ induce a $5$-cycle. By Lemma~\ref{l-c5twin}, no two vertices on this cycle have a common neighbor outside the cycle, contradicting the fact that both $u$ and $c$ are adjacent to $a$. This implies that $|\{a,b,c,v\}\cup N[\{d,q\}]|\geq 9$, so $|M(\{a,b,c\},\{d,q\})|\le 3^{(n-9)/3}$ due to Lemma~\ref{exc.inc.upperbound}. We conclude that $|M(\{a,b,c\},\{d\})|\le 3^{(n-7)/3}+3^{(n-9)/3}$, which is clearly upper bounded by $3^{(n-7)/3}+3^{(n-8)/3}$.
 
\medskip
\noindent
{\em Claim 5. $|M(\{a,b,c,d\},\emptyset)|\le 3^{(n-6)/3}$.}

\smallskip
\noindent
Every maximal induced matching in $G$ that does not cover any vertex in $\{a,b,c,d\}$ contains edge $uv$. Therefore, $M(\{a,b,c,d\},\emptyset)=M(\{a,b,c,d\},\{u,v\})$. Since $|N[\{u,v\}]|=6$ due to the fact that $G$ is cubic by Lemma~\ref{nodegree4}, it follows from Lemma~\ref{exc.inc.upperbound} that $|M(\{a,b,c,d\},\{u,v\})|\le 3^{(n-6)/3}$. This completes the proof of Claim 5.
 
\medskip
Combining the upper bounds in Claims 1--5 yields the following:
$$|M_G|\le 7\cdot 3^{\frac{n-6}{3}}+2\cdot 3^{\frac{n-7}{3}}+3^{\frac{n-8}{3}} < 3^{\frac{n}{3}}\,.$$
This contradicts the assumption that $G$ is a counterexample.
\qed 
\end{proof}

Lemma~\ref{lastlemma} states that $G$ is isomorphic to $K_{3,3}$, so in particular $n=6$. Since every maximal induced matching in $K_{3,3}$ consists of a single edge, we have that $|\mim{G}|=9=3^{n/3}$, contradicting the assumption that $G$ is a counterexample. This contradiction implies that any triangle-free graph on $n$ vertices has at most $3^{n/3}$ maximal induced matchings.

It remains to show that the bound in Theorem~\ref{t-main} is best possible. Let $G$ be the disjoint union of $p$ copies of $K_{3,3}$ for some positive integer $p$. Every maximal induced matching in $G$ contains exactly one edge of each connected component of $G$, which implies that $|\mim{G}|=9^p=9^{n/6}=3^{n/3}$. This completes the proof of Theorem~\ref{t-main}.

\end{document}